\documentclass[12pt]{amsart}
\usepackage{amsmath,amssymb,amsthm,amsfonts,bm,latexsym,graphicx,float,caption,subcaption,yfonts,cite,mathtools}
\graphicspath{{figures/}}
\usepackage[font=small,format=plain,labelfont=bf,textfont={footnotesize,it},justification=justified,singlelinecheck=true]{caption}
\usepackage[table,xcdraw]{xcolor}
\usepackage{multirow}

\def\R{\mathbb{R}}

\def\Z{\mathbb{Z}}
\def\rR{\mathcal{R}}
\def\I{\mathbb{I}}
\def\cI{\mathcal{I}}

\def\K{\mathcal{K}}
\def\J{\mathrm{J}}

\def\a{\alpha}

\newcommand{\eps}{\varepsilon}

\newtheorem{theorem}{Theorem}
\newtheorem{prop}[theorem]{Proposition}

\newtheorem{definition}[theorem]{Definition}

\everymath{\displaystyle}

\author{Guillaume Bal}
\address{Departments of Statistics and Mathematics and CCAM, University of Chicago, Chicago, IL 60637}
\email{guillaumebal@uchicago.edu}

\author{Fatma Terzioglu}
\address{Department of Statistics and CCAM, University of Chicago, Chicago, IL 60637}
\email{fterzioglu@uchicago.edu}

\date{}
\title[Uniqueness in multi-energy CT]{Uniqueness criteria in multi-energy CT}



\begin{document}
\maketitle

\begin{abstract}
  Multi-Energy Computed Tomography (ME-CT) is a medical imaging modality aiming to reconstruct the spatial density of materials from the attenuation properties of probing x-rays. For each line in two- or three-dimensional space, ME-CT measurements may be written as a nonlinear mapping from the integrals of the unknown densities of a finite number of materials along said line to an equal or larger number of energy-weighted integrals corresponding to different x-ray source energy spectra. ME-CT reconstructions may thus be decomposed as a two-step process: (i) reconstruct line integrals of the material densities from the available energy measurements; and (ii) reconstruct densities from their line integrals.

Step (ii) is the standard linear x-ray CT problem whose invertibility is well-known, so this paper focuses on step (i). We show that ME-CT admits stable, global inversion provided that (a well-chosen linear transform of) the differential of the transform in step (i) satisfies appropriate orientation constraints that makes it a $P-$matrix. We introduce a notion of quantitative $P-$ function that allows us to derive global stability results for ME-CT in the determined as well as over-determined (with more source energy spectra than the number of materials) cases. Numerical simulations based on standard material properties in imaging applications (of bone, water, contrast agents) and well accepted models of source energy spectra show that ME-CT is often (always in our simulations) either (i) non-globally injective because it is non-injective locally (differential not of full rank), or (ii) globally injective as soon as it is locally injective (differentials satisfy our proposed constraints).
\end{abstract}
\section{Introduction}
X-ray Computed Tomography (CT) is a well-known technique for visualizing the interior structure of an object of interest in a non-invasive manner. Measurement process involves irradiating the object cross-sectionally by x-ray beams which undergo photoelectric absorption at a degree depending on the material properties of the object, called the attenuation coefficient. This results in intensity loss in the x-ray beam which is recorded by a detector and processed by a computer to produce a two-dimensional image of x-ray attenuation map in each cross-section. A three-dimensional image of the object's internal structure may then be obtained by, for instance, combining the two-dimensional images of a series of parallel cross-sections acquired in multiple views; see, e.g., \cite{Barrett, Buzug, Herman}.

Without simplifying assumptions, the quantitative problem of image reconstruction in CT is a nonlinear inverse problem, with no known analytical solution. The standard forward model used in most CT applications employs Beer's Law by neglecting scattering (which we also do here) and assumes that the x-rays are monochromatic, i.e., have a fixed energy (which we do not want to do here). The image reconstruction then reduces to a linear inverse problem, which involves the recovery of the attenuation coefficient from its integrals along lines. However, in practice, x-ray beams are quite polychromatic (see left panel in fig.\ref{fig:spectrum_attenuation}), and the linear attenuation coefficient depends not only on the chemical composition of the object but also on the energy spectrum of the x-ray photons; see right panel in fig.\ref{fig:spectrum_attenuation}. Although the use of the linear approximation works well in general, for instance to determine the location of jumps in attenuation profiles, it is more qualitative than quantitative. Moreover, serious reconstruction errors may arise when the imaged object contains materials whose attenuation coefficients vary greatly with the energy level. These are the materials with high atomic numbers such as iodine, bone or metal  \cite{AlvarezMacovski, Lionheart, Katsura, Park}.

Dual-energy CT imaging leverages the energy dependence of attenuation to obtain material-specific information, as first proposed in the 1970s by Hounsfield \cite{Hounsfield}. It involves collecting measurements using two different x-ray energy spectra to identify two different materials in the imaged object based on their material density or atomic numbers \cite{AlvarezMacovski}. The advancement of photon counting detectors, which are capable of both counting the number of incoming photons and measuring photon energy, and thus collecting measurements simultaneously in more than two energy windows, made it feasible to differentiate three or more materials. This imaging modality is referred to as multi-energy CT (ME-CT), or spectral CT, imaging \cite{Leng, Schlomka, Taguchi, Willemink}. Advantages of dual- and multi-energy CT over standard CT are better tissue contrast and improved image quality at comparable or even reduced radiation doses by reducing beam hardening and metal artifacts. Current and emerging clinical applications of dual- and multi-energy CT include tissue characterization, lesion detection, oncologic imaging, vascular imaging and lately abdominal and musculoskeletal imaging. Existing reconstruction methods in dual- and multi-energy CT typically fall in to three categories: image-based  (e.g. \cite{Brooks, Maass}), projection-based (e.g. \cite{Abascal, Wu}) and one-step (e.g. \cite{Barber, Kazantsev, Long, Mechlem, Toivanen}) approaches (see also \cite{Mory, Vilches}). More information on dual- and multi-energy CT can be found, for example, in \cite{McCollough, Grajo} and their references.

Although material reconstructions in multi-energy CT are currently a very active research area, the analysis of their uniqueness and stability properties remains challenging.  Recent such analyzes are proposed in \cite{Alvarez2019,Levine}. The first reference  comes up with sufficient conditions beyond the non-vanishing of the Jacobian determinant of the dual-energy CT transform to guarantee uniqueness while the second one presents cases of non-uniqueness of the dual-energy CT transform. We are not aware of injectivity analyzes for general ME-CT problems. 

This paper proposes sufficient local criteria on the differential of the transform that guarantee a quantitative global injectivity of ME-CT. Such criteria obviously include the non-vanishing of the determinant of such a differential (the Jacobian determinant) to guarantee that the problem is locally injective. It is in fact not too difficult to come up with examples of ME-CT that are not injective locally.  It is also known that local injectivity does not imply global injectivity. Based on the work of \cite{GaleNikaido}, we show that local injectivity plus appropriate orientation constraints on the differential guarantee (quantitative) global injectivity. These constraints on the differential have to be verified numerically. A complete characterization even of when the Jacobian determinant remains globally positive still remains out of reach to-date.

The forward model of ME-CT is described in section 2. We then show in Theorem 1 of section 3 that the dual-energy CT transform is globally injective on a rectangle provided that the Jacobian determinant is nonvanishing everywhere. We then present sufficient criteria for global injectivity of more general ME-CT transforms in Theorem 2 using the theory of $P-$functions developed in \cite{GaleNikaido}. Extending the latter work, we obtain in the same section quantitative estimates of injectivity in the determined as well as the redundant measurement settings. 

Section 4 presents the results of numerical experiments for dual- and multi energy CT transform with two, three, and four commonly used materials and the corresponding number of energy measurements. These numerical experiments provide examples where the Jacobian determinant may vanish and change signs.  In all the examples we considered where the Jacobian determinant remains positive throughout the domain, we obtained numerically that the local criteria we proposed were always satisfied. In contrast, the positive (quasi-)definiteness or the diagonal dominance of the differential, which are also known to be (more restrictive) sufficient criteria for global injectivity, were often not satisfied.

\section{The Forward multi-energy CT Model}
Let $\Omega \in \R^N$ for $N=2,3$ denote the spatial volume of the imaged object whose material composition we want to identify. Following a standard approach \cite{AlvarezMacovski}, we assume that the linear attenuation coefficient $\mu(y,E)$ of the object at a point $y \in \Omega$ and at energy $E$ can be decomposed into a linear combination of functions of energy and functions of space such that
\[
\mu(E,y) = \sum_{j=1}^m M_j(E)\rho_j(y).
\]
Here, $m$ is the number of different materials, $M_j(E)$ the energy-dependent mass attenuation, also called basis function, of the $j$-th material, which is a known quantity (see right panel in fig. \ref{fig:spectrum_attenuation}), and $\rho_j(y)$ is the spatially-dependent mass density of the $j$-th material we want to recover. We define $M(E)=(M_j(E))_{1\leq j\leq m}$ and $x(l)=(x_j(l))_{1 \leq j \leq m}$ where $x_j(l)= \textstyle \int_l \rho_j dl$ denotes the x-ray transform of $\rho_j$ along a line $l$. 

For $1 \leq i \leq n$, let $S_i(E)$ denote the (known) product of the x-ray source energy spectrum and the detector response function for the $i$-th energy spectrum; see left panel in fig. \ref{fig:spectrum_attenuation}. We assume that the source/detector models $S_i$ are normalized so that $\textstyle \int_0^\infty S_i(E) dE=1$.

We consider measurements of the form
\begin{align}
   \int_0^\infty S_i(E) e^{-\int_l \mu(y,E) dy} dE = \int_0^\infty S_i(E) e^{-M(E)\cdot x(l)} dE,\qquad 1 \leq i\leq n.
\end{align}

We assume that $M_j(E)\geq 0$ and $x_j(l)\geq0$ for all $1\leq j \leq m,$ so the above physical measurements are between $0$ and $1$.

Then, the transform $I : \rR \subset \R^m \to \R^n$ modeling the second step of ME-CT measurements is defined as
\begin{align}
   I(x) = (I_i(x))_{1\leq i\leq n}, \qquad I_i(x) = -\ln  \int_0^\infty S_i(E) e^{-M(E)\cdot x} dE \geq0.
\end{align}
We assume here that the line integrals of interest $x=x(l)\in \rR \subset \R^m$.

Multi-energy CT measurements may thus be seen as the composition of the x-ray transform, which is linear and is well-studied, and a nonlinear map that performs different weighted averaging of x-ray projections over the energy range. Therefore, the reconstruction process typically consists of two steps: first, a nonlinear material decomposition reconstructing $x=x(l)$ from $I(x)$ for each line $l$; and second a linear tomographic reconstruction for each material density from its line integrals. This paper focuses on the first step. 
\begin{figure}[t]
\begin{center}
   \includegraphics[width=\textwidth]{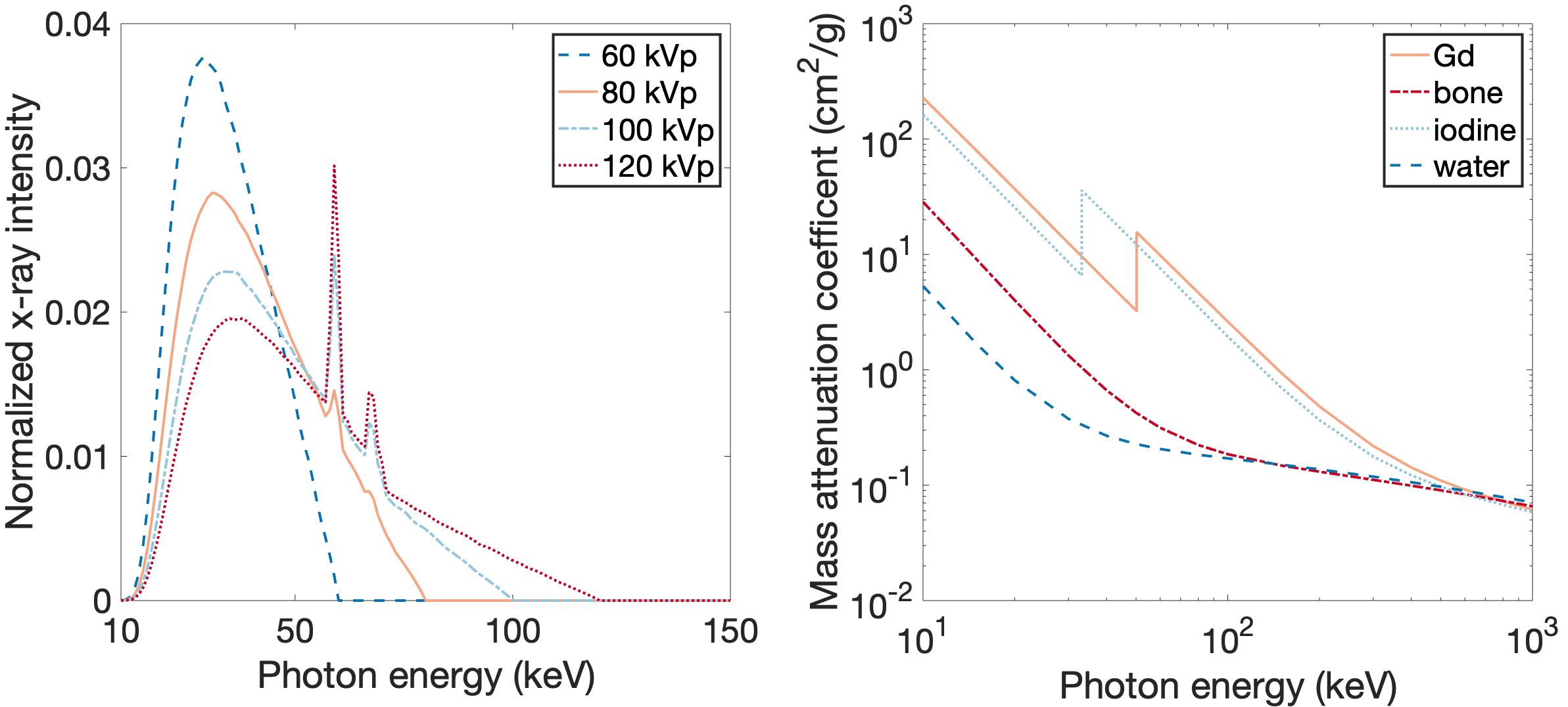}
    \caption{Left: Examples of x-ray source spectrum for varying tube potentials computed using the publicly available code SPEKTR 3.0 \cite{Spektr3}, and then normalized. Right: The x-ray attenuation coefficients of gadolinium, bone, iodine and water as functions of x-ray energy in log-log scale. The raw data was obtained from NIST \cite{NIST}. }
    \label{fig:spectrum_attenuation}
\end{center}
\end{figure}
For the rest of the paper, we are thus interested in the injectivity of the mapping $x  \in \rR \mapsto I(x) \in \R^n$ where $\rR \subset \R^m$, for technical reasons, is chosen as a closed rectangle (a Cartesian product of closed intervals). The map $I$ is smooth for $S_i$ compactly supported and $M_j$ bounded, which we now assume, and its Jacobian at a point $x \in \rR$ is given explicitly by the matrix $J(x)$ with coefficients
\begin{align}
  J_{ij}(x) = e^{I_i(x)} \int_0^\infty S_i(E) M_j(E) e^{-M(E)\cdot x} dE ,\qquad 1 \leq i \leq n, \;1 \leq j \leq m.
\end{align}
Clearly, all entries of the Jacobian matrix are strictly positive. Thus, the map $I$ is strictly isotone, that is, for any $x, a \in \Omega$,  $I(x) > I(a)$  whenever $x > a$ \cite{Rheinboldt}. In the following, both the matrix $J$ and (if applicable) its determinant will be referred to simply as the Jacobian when the difference is clear from the context. The notation $x> 0$ means that all coordinates of the vector $x$ are positive and $x>a$ means that $x-a>0$.  Similarly, the inequality $A> 0$ means that all elements of the matrix $A$ are positive. We will use the symbol $\I$ to denote the $n \times n$ identity matrix.
 
\section{Injectivity of multi-energy CT Transform}
We are interested in (sufficient) criteria that guarantee the injectivity of the map $x \mapsto I(x)$. We first consider the case $m=n$. A necessary condition for local injectivity is that det $J(x) \neq 0$ (inverse function theorem). However, the non-vanishing of the Jacobian is clearly not sufficient in general (although it is for the specific application of dual energy CT with $n=2$ as we show later in this section).
 
 The Hadamard global inverse function theorem \cite{Hadamard} states that a differentiable map $F:\R^n \to \R^n$ with a nonvanishing Jacobian is a diffeomorphism if and only if $F$ is proper, that is in this context, $\lim_{|x| \to \infty}|F(x)| = \infty$. This result, which is topological in nature, does not provide any quantitative estimates of injectivity. 

When the domain of $F$ is a rectangular region $\rR \subset \R^n$, a sufficient criterion for global injectivity based on the notion of $P$-functions and due to Gale and Nikaido \cite{GaleNikaido} reads as follows:\\

\mbox{%
    \parbox{0.9\textwidth}{%
\emph{Let $F : \rR \subset \R^n \to \R^n$ be differentiable on the closed rectangle $\rR$. If the Jacobian $J(x)$ of $F$ is a $P-$matrix for each $x \in \rR$, then $F$ univalent (injective) in $\rR$.}
    }%
    }\\
    
A matrix $A$ is called a $P-$matrix if all principal minors of $A$ are positive. Principal minors of a $n\times n$ matrix $A$ are defined as follows. Let $K$ and $L$ be subsets of $\{1,...,n\}$ with $k$ elements. The minor of $A$ associated to $K$ and $L$, denoted by $[A]_{K,L}$, is the determinant of the $k \times k$ submatrix of $A$ formed by deleting all the rows with index in $K$ and columns with index in $L$. If $K=L$, then $[A]_{K,L}=[A]_K$ is called a principal minor. A function whose differential is a $P-$matrix is called a $P-$function (see below for an equivalent definition).

The positivity of all principal minors is intimately related to the preservation of orientation. The following related geometric characterization for $P-$matrices will be useful in the sequel: \cite{FiedlerPtak, GaleNikaido}\\

\mbox{%
    \parbox{0.9\textwidth}{%
\emph{$A$ is a $P-$matrix if and only if $A$ reverses the sign of no vector except zero, that is to every nonzero vector $x$ there exists an index $i$ such that $x_i(Ax)_i> 0$}.
    }%
    }\\
    
In fact, when the map $F$ is continuously differentiable on $\rR$, it is sufficient that the Jacobian be positive everywhere and a $P-$matrix only at the boundary $\partial \rR$. This result, which combines the orientation preservation at the domain's boundary and a topological argument similar to that leading to the Hadamard univalence theorem, was proven independently (in slightly different forms) by Mas-Colell \cite{MasColell}, and Garcia and Zangwill \cite{Garcia}. We refer the reader to intuitive examples in \cite{MasColell} showing why the orientation preservation is only sufficient for injectivity, but not necessary. The main result of this paper is to apply a modified version of the orientation preservation results of \cite{GaleNikaido} to numerically prove that ME-CT is injective in many cases of practical interest. 
	
Another sufficient criterion, also in \cite{GaleNikaido}, states that if the Jacobian matrix is positive (negative) quasi-definite\footnote{A is said to be positive (negative) quasi-definite, if its symmetric part, namely  $(A + A^\top)/2$ is positive (negative) definite.}, then univalence holds not only on rectangular but on any convex region.
Positive quasi-definite matrices as well as strictly diagonally dominant matrices \footnote{A matrix $A=[a_{ij}]_{i,j=1}^n$ is strictly diagonally dominant if $|a_{ii}| > \sum_{j\neq i} |a_{ij}|$ for each $i=1,...,n$.} having positive diagonal entries are subclasses of P-matrices \cite{GaleNikaido}. However, from an algorithmic point of view, the latter are significantly better than orientation-preserving: If the Jacobian matrix is positive quasi-definite or strictly diagonally dominant everywhere, then iterative algorithms such as Gauss-Seidel provably converge to the global inverse \cite{More,Frommer}. In the case of $P-$matrix Jacobians, no algorithm in the literature is guaranteed to converge to the global inverse. Our numerical experience with ME-CT is that the differentials are $P-$matrices that are neither positive quasi-definite nor diagonally dominant. This provides the usefulness of the notion of $P-$functions in ME-CT.

Here, we prove that the nonvanishing of the Jacobian determinant is not only necessary for the injectivity of the dual-energy CT transform (the case $n=2$) but is also sufficient.
\begin{theorem}[{dual-energy CT-injectivity}]\label{dual-energy CT-injectivity}
Let $\rR \subset \R_+^2$ be a rectangular domain. The dual-energy CT transform  $I: \rR \to \R^2$ is injective if its Jacobian never vanishes in $\rR$.
\end{theorem}
\begin{proof}
 It is not true in general that functions from $\R^2$ to $\R^2$ with everywhere positive Jacobian are necessarily injective. A counter-example is given in \cite{GaleNikaido}. However, we are here in a setting where all entries of the Jacobian $J(x)$ are positive, while the principal minors of $J$ are the diagonal entries $J_{11}$ and $J_{22}$, and det $J$. Hence, if  det $J$ is positive everywhere in $\rR$, then J is a $P-$matrix. When the det $J$ is negative throughout $\rR$, exchanging the two rows in $I(x)$, which is an invertible transformation, leads to a sign change in det $J$. We then apply the above Gale-Nikaido theorem to obtain the result. 
\end{proof}

In the case of multi-energy CT, a sufficient condition for global injectivity, which is a direct consequence of the Gale-Nikaido theorem \cite{GaleNikaido}, is as follows.
\begin{theorem}[{multi-energy CT-injectivity}]\label{multi-energy CT-injectivity}
Let $x \mapsto I(x)$ on a closed rectangle $\rR \subset \R_+^n$. If the Jacobian $J(x)$ is a $P-$matrix for all $x \in \rR$, then $I$ is injective (univalent) in $\rR$.
\end{theorem}
\begin{proof} 
The proof is clear as the mapping $x \mapsto I(x)$ is differentiable as required in the Gale-Nikaido theorem \cite{GaleNikaido}.
\end{proof}

\subsection{\textbf{Transforming $\bm{I}$ linearly into a $P-$function}} 
Since having a $P-$matrix Jacobian is only a sufficient criterion for injectivity, the map $I$ can still be injective even though its Jacobian is not a $P-$matrix. In fact, in our numerical experiments, the multi-energy CT transform proved to be injective as soon as its Jacobian never vanished in the rectangle $\rR$. However, what is a $P-$matrix is not the differential of $I$ itself but rather a linear modification of it.

A map $F:\rR \subset \R^n \to \R^n$ is called a $P-$function if for any $x, y \in \rR, x \neq y$, there exists an index $k=k(x,y)$ such that 
$$(x_k-y_k)(f_k(x)-f_k(y))>0.$$
Here $x_k$ and $f_k(x)$ are the $k$-th components of $x$ and $F(x)$, respectively \cite{MoreRheinboldt}. It is known that if $F$ is a $P-$function if and only if it is injective and its inverse $F^{-1}$ is also a $P-$function \cite{MoreRheinboldt,Rheinboldt}. Moreover, a differentiable map on a rectangle $\rR$, whose Jacobian is a $P-$matrix everywhere in $\rR$, is a $P-$function {\cite{GaleNikaido, MoreRheinboldt}.

In the case that $I$ is not a $P-$function, one way to prove injectivity is to map $I$ into a $P-$function via an invertible linear transformation, because then the invertibility of $I$ and the new map are equivalent. This simple fact is proven below.
\begin{prop}\label{TransformI}
Let $I, \tilde{I}: \rR \subset \R^n \to \R^n$ be two maps such that $\tilde{I} = \mathcal{A} \circ I$ where $\mathcal{A}: \R^n \to \R^n$ is an invertible linear transformation. Then, $\tilde{I}$ is injective in $\rR$ if and only if $I$ is injective in $\rR$.
\end{prop}
\begin{proof} 
It is immediate: Observe that 
$$\tilde{I}(x)-\tilde{I}(x') =  (\mathcal{A} \circ I) (x) -  (\mathcal{A} \circ I)(x') = A(I(x)-I(x')),$$
where $A$ is the matrix of $\mathcal{A}$, that is $\mathcal{A}(y) = Ay$.
Thus, if $x, x' \in \rR$, $x \neq x'$, then $\tilde{I}(x) \neq \tilde{I}(x')$ iff $I(x) \neq I(x')$ as $null(A) = \{ 0\}$.
\end{proof}
Transforming $I$ linearly into a $P-$function is equivalent to finding a matrix $A$ (independent of $x$) such that $AJ(x)$ is a $P-$matrix for all $x \in \rR$. 
 
\subsection{Quantitative $P-$functions}
The injectivity results obtained from the work in \cite{GaleNikaido} are not quantitative, and thus cannot be applied directly to derive stability estimates for the reconstructions. 
We first obtain the following extension.
\begin{prop}\label{Qprop}
If $A$ is a $P-$matrix, then there is $\mu>0$ such that $A-\mu \I$ is a $P-$matrix.
\end{prop}
\begin{proof}
Assume for a contradiction that there is no $\mu>0$ such that $A-\mu \I$ is a $P-$matrix. This means that for all $\mu>0$ there is a nonzero vector $u$ such that $u_i(Au-\mu u)_i \leq 0$ for all $i=1,\dots, n$. Then, we can find a sequence $u^{\{j\}}$ with $\|u^{\{j\}}\|=1$ such that $u^{\{j\}}_i(Au^{\{j\}})_i \leq \mu_j (u^{\{j\}}_i)^2$ for all $i=1,\dots, n$ with $\mu_j\to0$. On the unit sphere, we find a subsequence, still called $u^{\{j\}}$, converging to $v$. Now $\|v\|=1$, and by continuity, $v_i(Av)_i \leq 0$ for all $i=1,\dots, n$. Therefore $A$ reverses the sign of $v$, which leads to a contradiction as $A$ is a $P-$matrix.
\end{proof}

Note that this simply shows that if $A$ is a $P-$matrix, then $A-\lambda \I$ is still a $P-$matrix for all $0 \leq \lambda \leq \mu$.
\begin{definition}
  Let $I$ be a $P-$function on a closed rectangle $\rR$ with continuous Jacobian $x\mapsto J(x)$. We define
\[
		\mu := \max_{\lambda>0}\{I(x)-\lambda x \text{ is a $P-$function on } \rR \}.
\]
We call $\mu=\mu(I)$ the injectivity constant of $I$. 
\end{definition}

Note that by the preceding proposition, $\mu=\mu(x)>0$ exists for all $x\in\rR$ and by continuity and compactness, there is a largest such lower bound $\mu>0$. We now have the following quantitative reformulation of \cite[Theorem 3]{GaleNikaido}.  
\begin{prop}
  Let $I(x)$ be a $P-$function on $\rR$ with continuous Jacobian $x\mapsto J(x)$ and injectivity constant $\mu$. Let $0\leq \lambda \leq \mu$,  $a\in \rR$ and define
\[
   X =\big\{  x\geq a,\quad I(x)-I(a) \leq \lambda (x-a) \big\}.
\]
Then $X=\{a\}$.
\end{prop}
\begin{proof}
Let $\mu$ and $\lambda$ be given as in the theorem. We know from Proposition \ref{Qprop} that $I(x)-\lambda x$ is a $P-$function, and thus $J(x)-\lambda \I$ is a $P-$matrix at every $x \in \rR$. Thus, the application of Theorem 3 in \cite{GaleNikaido} yields the result. 
\end{proof}

We can now obtain the following quantitative estimate of injectivity.
\begin{theorem}\label{invI is Lipschitz}
  Let $I(x)$ and $\lambda$ as in the preceding Proposition.  Then, for all $x$ and $a$ in $\rR$, we have that 
\[
  |I_i(x)-I_i(a)| \geq \lambda |x_i-a_i|.
\]
Therefore, in any $l^p$ norm, $\|I(x)-I(a)\|_p\geq \lambda \|x-a\|_p$. 
\end{theorem}
\begin{proof}
  For $x\geq a$, we observed that $I(x)-I(a)\geq \lambda(x-a)\geq0$ in the preceding theorem. This implies the above estimate. Now consider a diagonal change of variables $D:\R^n\to\R^n$, which to each variable $x_i$ associates $\pm x_i$. We verify that $D\circ I \circ D$ is a $P-$function as an immediate property of minors of $P-$matrices. For any pair of elements $(x,a)$ in $\rR$, we find a $D$ such that $Dx\geq Da$. Therefore, $DI_i(x)-DI_i(a)\geq \lambda (Dx-Da)\geq0$ and hence   $|I_i(x)-I_i(a)| \geq \lambda |x_i-a_i|$. This proves the quantitative injectivity result.
\end{proof}

Note that on its range $I(\rR)$, the function $I$ is injective and hence invertible. The inverse is then a $P-$function as is well known and is moreover Lipschitz in the sense that $|I_i^{-1}(x) - I^{-1}_i(a)| \leq \lambda^{-1} |x_i-a_i|$.

Let us remark that in the above proofs, all we show is that $I-\lambda x$ is a $P-$function. This comes from the fact that $J(x)-\lambda \I$ is a $P-$matrix at every $x \in \rR$.

\subsection{The case of redundant measurements}
We now consider the case $m > n$ with more energy measurements than unknown material densities. 

The global injectivity in the redundant setting is fairly similar to the determined case (where $m=n$) in the sense that global properties cannot come from local ones. One can easily construct two one-dimensional functions such that, at each point, at least one of them have positive derivative, and yet the two functions can meet at several points. For example, let $I : [0,1] \to \R^2$, $I(x)=(f(x),g(x))$ with
\begin{align*}
f(x)=
   \begin{cases} 
     x, & x \in [0,\tfrac23], \\
     2-2x,  & x \in [\tfrac23,1],\\
   \end{cases}
   \quad \text{and} \quad  g(x)=
   \begin{cases} 
     -2x, & x \in [0,\tfrac13], \\
     x-1,  & x \in [\tfrac13,1].
   \end{cases}
\end{align*}
 The best available derivative from both functions equals $1$ throughout the interval $(0,1)$ and yet $f(0)=f(1)=g(0)=g(1)=0$, so injectivity for the family $\{f,g\}$ does not hold. One of the functions therefore must handle global injectivity in part of the domain without influence from the other one.  
 
 Consider $I : \R \to \R^2$, $I(x)=(f(x),g(x))$ with
\begin{align*}
f(x)=
   \begin{cases} 
     k, & x \in [2k-1,2k], \\
     x-k,  & x \in [2k,2k+1],\\
   \end{cases}
    \quad \text{and} \quad g(x) = x-f(x), \quad k \in \Z.
\end{align*}
Now $f$ has derivative $1$ on the intervals $[2k,2k+1]$ and derivative $0$ on the intervals $[2k-1,2k]$ while $g$ has derivatives $0$ and $1$ on these intervals, respectively. Combined, we find an injectivity constant $\mu = 1$ and an effective constant $\tfrac{\mu}{|{\K|}}=\tfrac12$, while the injectivity constant of each function individually is $0$. Replacing the above $0$ derivatives by $\eps$, we are in the setting of the above result with two bona-fide $P-$functions that collectively provide much better stability than individually.

 This behavior is prevented by assuming that all functions of interest $I_K$ are $P-$functions throughout the domain of interest. Although we do not pursue here, one can certainly consider generalizations where $x\mapsto I_K(x)-\mu x$ is a $P-$function on some part of the domain while $x\mapsto I_K(x)+\mu x$ is a $P-$function on other parts of the domain, with $\mu \geq 0$ sufficiently small that injectivity is still achieved. 
 
 The above notions generalize to the multi-dimensional setting. Let 
 $$ \K = \{ K \subset \{1,\dots,m\} \; | \; |K|=n \}. $$
Clearly, $ |\K| =\tbinom{m}{n}$. For a given $I:\rR\subset \R^n \to \R^m$, for each $K\in{\K}$, we denote by $I_K:\rR\subset\R^n\to\R^n$ the corresponding subsystem. 

While global injectivity for a given pair of points $(x,a)$ has to be obtained from a fixed subsystem, that system may vary for different pairs of points $(x,a)$. 
\begin{definition}
  We say that $\{ I_K : K\in {\K}\}$ is a $P-$family with injectivity constant $\mu>0$ if 
  \begin{enumerate}
  \item for each $K\in \K$, $I_K$ is a $P-$function on $\rR$, and
  \item there is a cover of $\rR$ by rectangles $U_{\a}$, $\a \in{\cI}$, such that for each $\a \in{\cI}$, there exists $K\in{\K}$ with $I_K-\mu x$ being a $P-$function on $U_{\a}$.
  \end{enumerate}
\end{definition}
The following result is an analog of Theorem \ref{invI is Lipschitz} in the redundant measurement setting.
\begin{theorem}
  Let  $\{ I_K : K\in {\K}\}$ be a $P-$family with injectivity constant $\mu>0$. For any $a,x$ in $\rR$, if $[a,x]$ is the line segment joining $a$ and $x$, then $[a,x] \subset \cup_{j=0}^{k-1} U_{\a_j}$ for some $k\geq 1$ and $\a_j \in{\cI}$. Let 
$
  \K' = \{K \in \K \; | \: I_{K}-\mu x \text{ is a } P-\text{function on } U_{\a_j} \text{ for some } 0\leq j\leq k-1\}.
$
Then, we have
\[
  \frac{1}{|\K'|}\Big\|\sum_{K \in \K'} I_{K}(x)-I_{K}(a) \Big\| \geq \left(\frac{\mu-\mu_0}{|\K'|}+\mu_0 \right)\|x-a\|,
\]
where $\mu_0 = \min_{K \in \K'}\mu(I_K)$.
\end{theorem}
\begin{proof}
Consider two points $x\geq a$ in $\rR$. Since $\{ I_K : K\in {\K}\}$ is a $P-$family with injectivity constant $\mu>0$, there is a cover of $\rR$ by rectangles $U_{\a}$, $\a \in{\cI}$, which may assumed to be closed, such that for each $\a \in{\cI}$, there exists $K\in{\K}$ with $I_K$ and $I_K-\mu x$ being $P-$functions on $\rR$ and $U_{\a}$, respectively. Now $[a,x] \subset \cup_{j=0}^{k-1} U_{\a_j}$ for some $k\geq 1$, $\a_j\in{\cI}$, and there exist points $a=y_0\leq y_1\leq\ldots \leq y_k=x$ such that $\{y_j,y_{j+1}\}\in U_{\a_j}$ for $\a_j\in{\cI}$ and $0\leq j\leq k-1$.  Since for each $0\leq j\leq k-1$, there is a $K\in{\K'} \subseteq \K$ such that $I_{K}$ and $I_{K}-\mu x$ are $P-$functions on $\rR$ and $U_{\a_j}$, respectively, we can apply Theorem \ref{invI is Lipschitz} to obtain
\begin{align*}
\sum_{K \in \K'} I_{K}(x)-I_{K}(a) &= \sum_{K \in \K'} \sum_{j=0}^{k-1}I_{K}(y_{j+1})-I_{K}(y_j) \\
 &= \sum_{j=0}^{k-1} \sum_{K \in \K'} I_{K}(y_{j+1})-I_{K}(y_j) \\
&\geq \sum_{j=0}^{k-1}  \mu (y_{j+1}-y_j) + (|\K'|-1)\mu_0 (y_{j+1}-y_j)\\
&=((\mu-\mu_0) +  |\K'|\mu_0)(x-a).
\end{align*}
Therefore,
\[
  \frac{1}{|\K'|}\sum_{K \in \K'} I_{K}(x)-I_{K}(a) \geq \left(\frac{\mu-\mu_0}{|\K'|}+\mu_0 \right)(x-a),
\]
which implies the result for $x\geq a$. 

Now for each pair $(a,x)$ in $\rR$, there is  a diagonal change of variables $D:\R^n\to\R^n$ as before such that $Dx\geq Da$. Moreover, $D\circ I_K \circ D$ remains a $P-$function on $\rR$ while $D\circ(I_K-\mu x) \circ D$ is a $P-$function on $D\rR_{\a}$, if $I_K$ and $I_K-\mu x$ are $P-$functions on $\rR$ and $\rR_{\a}$, respectively. We then apply the same decomposition as above in that new set of variables to get the estimate. 
\end{proof}
Let us now consider a set $\K'$ that is independent of the segment $[a,x]$ (we can always find such a set). The above results states that we may replace the measurements by the average $\tilde I(x)=\frac{1}{|\K'|}\sum_{K\in\K'} I_K(x)$. We then find that $\tilde I$, which may be constructed from available measurements is invertible and its inverse has a Lipschitz constant bounded by $(\frac{\mu-\mu_0}{|\K'|}+\mu_0)^{-1}$, which may be much smaller than $\mu_0^{-1}$ if $\mu$ is larger than $\mu_0$ and $|\K'|$ can be kept sufficiently small.

\section{Numerical Experiments}
In this section, we present the results of our numerical experiments for dual- and multi energy CT transform with two, three, and four commonly used materials and the corresponding number of energy measurements. In each case, for a fixed set of materials, we provide examples where the local and/or global injectivity of the problem is guaranteed. 

In the following, the below parameters were used:
\begin{itemize}
\item The diagnostic energy range $10\leq E\leq150$ (keV) was considered.
\item The energy spectra $S_i, \; i=1,\dots n,$ corresponding to given tube potentials $tp_i$ were computed using the publicly available code SPEKTR 3.0 \cite{Spektr3}. For practical purposes, integer valued tube potentials ranging from 40-150 kVp were considered. We denote $tp = (tp_1, \dots, tp_n)$. 
\item  The domain of the transform $I$:
\[
\rR = \Bigg\{ (x_1,\dots,x_n) \in \R_+^n : \; 0 \leq x_j \leq \frac{10}{\displaystyle \max_{10\leq E \leq 150} M_j(E)} \Bigg\},
\]
where $M_j(E)$ denotes the energy-dependent mass-attenuation of the $j$-th material, and $M(E) = (M_j(E))_{1\leq j\leq n}$.
We note that then $e^{-M(E)\cdot x} \geq e^{-10}$, which is even more conservative than practically relevant rectangle size.
\end{itemize}

\subsection{Two Materials-Two Measurements Case}
In view of theorem \ref{dual-energy CT-injectivity}, dual-energy CT problem is globally injective if the Jacobian never vanishes inside $\rR$. Considering two commonly used material pairs, namely (bone, water) and (iodine, water) in the said order, we tested whether the Jacobian can vanish inside $\rR$ for integer valued tube potentials varying from 40-150 kVp. 

For  (bone, water) material pair, there was no case of Jacobian vanishing inside $\rR$. On the other hand, for (iodine, water) pair, the probability of encountering a vanishing Jacobian was 22\%. The tube potential pairs $tp = (tp_1,tp_2)$ that lead to Jacobian vanishing inside $\rR$ is shown in the left panel of fig. \ref{fig:TP4dect}. We note that decreasing the density of iodine has no influence on vanishing of the Jacobian. It only makes it to have smaller values.

We also searched for tube potentials that lead to everywhere diagonally dominant or positive quasi-definite Jacobian matrix. For (bone, water) material pair, there was no case of diagonally dominant Jacobian, but the probability of finding a positive quasi-definite Jacobian was 6\%. The tube potential pairs $tp = (tp_1,tp_2)$ that lead to positive quasi-definite Jacobian matrix inside $\rR$ is depicted in the right panel of fig. \ref{fig:TP4dect}. However, for the (iodine,water) material pair, there was neither a case of the Jacobian being diagonally dominant nor being positive quasi-definite everywhere even when we changed the density of iodine.

\begin{figure}[h]
\begin{center}
   \includegraphics[width=\textwidth]{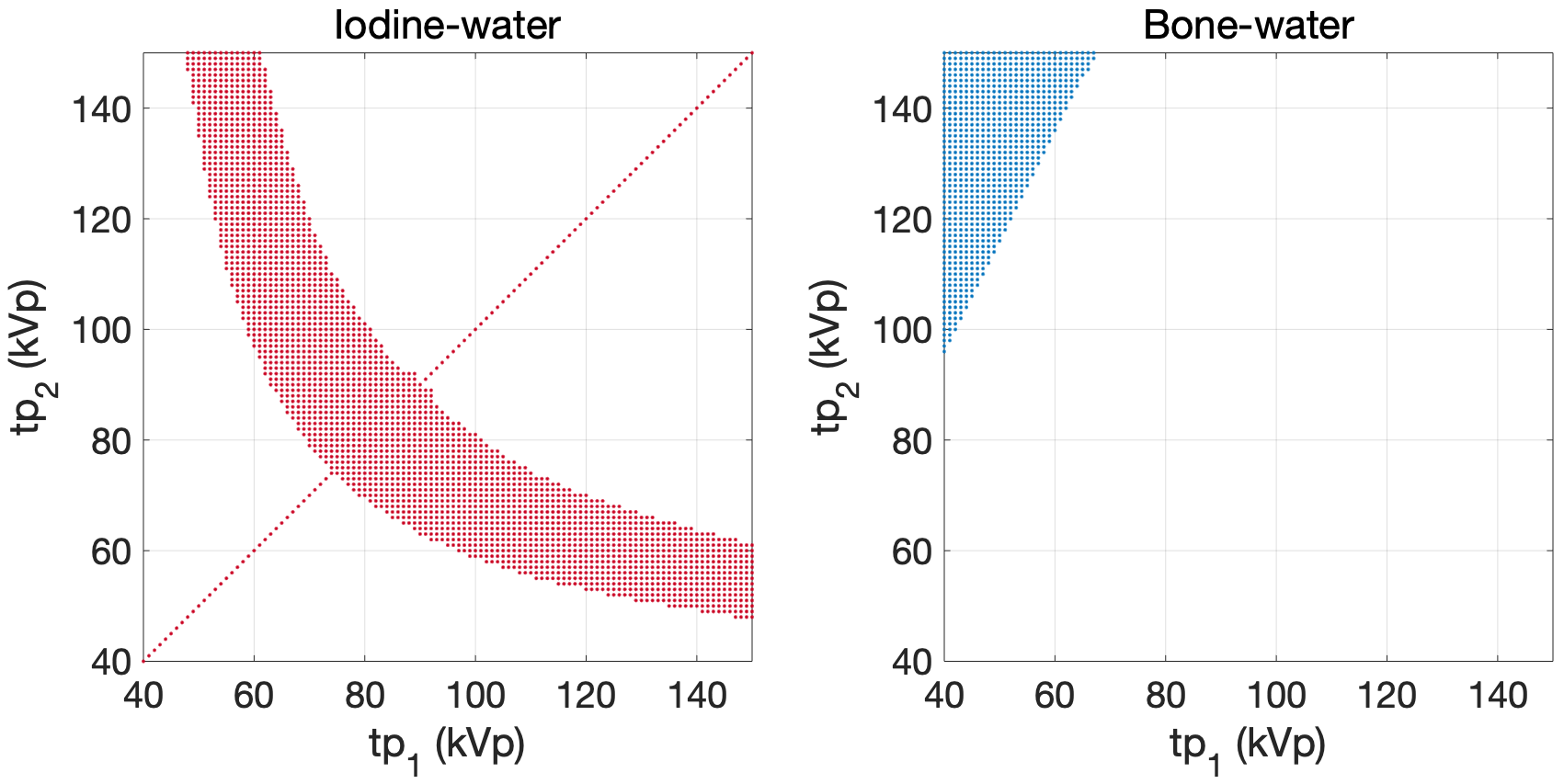}
    \caption{Left: The tube potential pairs $tp = (tp_1,tp_2)$ (in red) that lead to vanishing Jacobian inside $\rR$ for the material pair (iodine, water). Right: The tube potential pairs $tp = (tp_1,tp_2)$ (in blue) that lead to positive quasi-definite Jacobian matrix inside $\rR$ for the material pair (bone, water).}
    \label{fig:TP4dect}
\end{center}
\end{figure}

\subsection{Three Materials-Three Measurements Case}
In the following, we used a fixed set of materials (bone, iodine, water) in the said order. For varying tube potentials $tp$, we examined some phenomena that are related to the invertibility of the ME-CT transform. Below we present some representative examples.

\subsubsection{\textbf{The Jacobian can vanish inside the rectangle $\rR$}}
The probability of the Jacobian vanishing inside $\rR$ was around 4\%. (It is \%1 the spectra are separated, i.e. the tube potentials are different.) Fig. \ref{fig:TP40detJ3d} shows some tube potentials leading to vanishing Jacobian inside $\rR$. 
\begin{figure}[h]
\begin{center}
      \includegraphics[width=\textwidth]{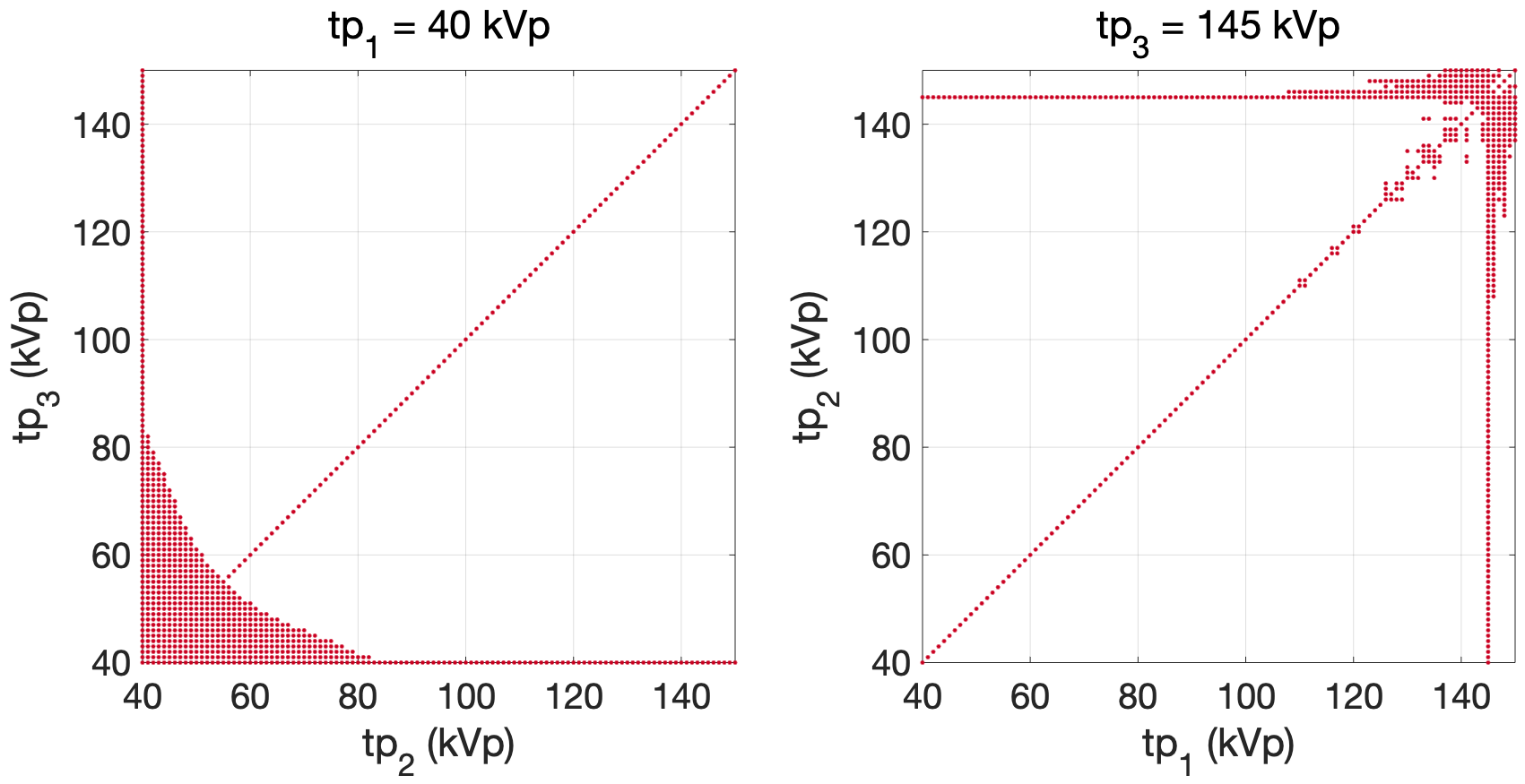}
    \caption{Some tube potentials $tp = (tp_1,tp_2,tp_3)$ that lead to vanishing Jacobian inside $\rR$ for the materials (bone, iodine, water). Scatter plots for $tp_2$ and $tp_3$ when $tp_1 = 40$ kVp(left) and for $tp_1$ and $tp_2$ when $tp_3 = 145$ kVp (right).}
        \label{fig:TP40detJ3d}
\end{center}
\end{figure}

 \subsubsection{\textbf{Transforming the mapping $\bm I$ into a $\bm P-$function when it is not so}} 
Transforming $I$ linearly into a $P-$function is equivalent to finding a matrix $A$ (independent of $x$) such that $AJ(x)$ is a $P-$matrix for all $x \in \rR$. According to our numerical experiments, this seems possible as long as $\det J(x)$ is nonvanishing everywhere in $\rR$. 

We considered several cases where $\det J(x) >0$ for all $x \in \rR$, but $J(x)$ is not a $P-$matrix for some $x \in \rR$. We note that the case $\det J(x) < 0$ for all $x \in \rR$ can be dealt with by exchanging two rows/columns of $J$. 

In view of proposition \ref{TransformI}, by using $10^6$ many random $3\times3$ matrices $A$ with $\det A=1$ (in order to keep the size of the volume fixed), we checked if $AJ(x)$ is a $P-$matrix for all $x \in \rR$. We observed that it is always possible to find a desired $A$, with a probability around 2-3\% (notice that the existence of even one such $A$ guarantees injectivity). One such example is $tp = (40, 60, 140)$, which leads to a Jacobian that is not a $P-$matrix everywhere in $\rR$ (see table \ref{tbl:minors4notPJacobian}). Among all $A$'s such that $AJ(x)$ is a $P-$matrix for all $x \in \rR$, the bigest injectivity constant obtained was $\mu = 0.2306$ when 
 \begin{align}\label{QuantitativeA4notP}
A=
\begin{bmatrix}
    \phantom{-}1.0542  &  -0.2669  & -0.8656\\
   -0.3485  &  \phantom{-}1.1163  &  -0.8111\\
   \phantom{-} 0.6081   &  \phantom{-}1.6056  &  -0.9398
\end{bmatrix}.
\end{align}
\begin{table}[h]
\begin{tabular}{|c|c|c|c|c|}
\hline
\multicolumn{5}{|c|}{tp = {(}40, 60, 140{)}} \\ \hline
\multirow{2}{*}{\begin{tabular}[c]{@{}c@{}}minor\\ assoc. to \end{tabular}} & \multicolumn{2}{c|}{minors of $J$} & \multicolumn{2}{c|}{minors of $AJ$} \\ \cline{2-5} 
& min              & max             & min              & max              \\ \hline
$\O$    & 0.8931 & 1.2797  &  0.8931 & 1.2797   \\ \hline
$\{1\}$ &  -0.0833 &  0.5299  &  1.4873 &  2.5446     \\ \hline
$\{2\}$ & 0.3128 & 0.5863 &  0.4145 & 0.7904      \\ \hline
$\{3\}$ & 15.890 & 25.794 &  0.8596 & 1.8397      \\ \hline
$\{2,3\}$ & 1.5097 &  2.7714  &  0.8888 & 1.6883      \\ \hline
$\{1,3\}$ &  12.861 &  16.132 &  0.2306 &  2.4807      \\ \hline
$\{1,2\}$ &  0.2163 &  0.2763 &   0.6443 &  0.9280    \\ \hline
\end{tabular}
\vspace{1em}
\caption{Minimum and maximum values of the minors of $J$ and $AJ$ attained in the rectangle $\rR$  for $tp = (40, 60, 140)$ and $A$ given as in \eqref{QuantitativeA4notP}.}
\label{tbl:minors4notPJacobian}
\end{table}
We also observed that for tube potentials that are in increasing order and at least 10 keV apart, the reason for $J$ not being a $P-$matrix was only one of the 2-minors becoming negative. This makes it possible to increase the probability of finding a desired $A$ up to around 10\% when the random $A$ matrices are chosen adaptively as we explain now. If $P-$matrix condition is violated because $[J]_{\{i\}}$ is nonpositive somewhere in $\rR$  for some $i=1,2,3$, the random $A$ matrices can be chosen from the set of real matrices 
$$\mathcal{M}_i := \{ A=[a_{kl}]_{k,l=1}^3 \;|\; a_{kk}=1 \text{ and } a_{kl}=0 \text{ for } k\neq l \neq i\}.$$
Then, for $j\neq i$, $[AJ]_{\{j\}}=[J]_{\{j\}}$. This is due to the following fact about the minors of product of two matrices. Suppose that $A$ and $J$ are $n\times n$ matrices, and $K$ and $L$ are subsets of $\{1,...,n\}$ with $k$ elements. Then,
\begin{align}\label{Cauchy-Binet}
[AJ]_{K,L} = \sum_{M}[A]_{K,M}[J]_{M,L}
\end{align}
where the summation runs over all subsets $M$ of $\{1,...,n\}$ with $k$ elements.
This formula is a generalization of the formula for ordinary matrix multiplication and the Cauchy-Binet formula for the determinant of the product of two matrices.

We finally note that if more than one 2-minor of $J$ were nonpositive, not necessarily at the same $x \in \rR$, then one could successively multiply $J$ with a suitable $A \in \mathcal{M}_i$ to have an everywhere $P-$matrix Jacobian.

\subsubsection{\textbf{Transforming a barely $\bm P-$function, into a quantitative $\bm P-$function}}
For example, when $tp = (50, 75, 110)$, the Jacobian is a $P-$matrix but the injectivity constant is equal to 0.0001. By using a linear transformation with matrix
\begin{align}\label{QuantitativeA4barelyP}
A=
\begin{bmatrix}
\phantom{-}0.7067 & -0.1425 & -0.8578\\
-0.2656 & \phantom{-}0.8144 & -0.5679\\
\phantom{-}0.7319 & \phantom{-}0.5258 & \phantom{-}0.1835\\
\end{bmatrix},
\end{align}
we obtained a $P-$function with injectivity constant equal to 0.3190 (see table \ref{tbl:minors4barelyPJacobian}).
\begin{table}[h]
\begin{tabular}{|c|c|c|c|c|}
\hline
\multicolumn{5}{|c|}{tp = {(}50, 75, 110{)}} \\ \hline
\multirow{2}{*}{\begin{tabular}[c]{@{}c@{}}minor\\ assoc. to \end{tabular}} & \multicolumn{2}{c|}{minors of $J$} & \multicolumn{2}{c|}{minors of $AJ$} \\ \cline{2-5} 
& min              & max             & min              & max              \\ \hline
$\O$    & 0.2408 & 0.3012  &  0.2408 & 0.3012    \\ \hline
$\{1\}$ & 0.0017 &  0.3154  &  0.3952 & 0.5292      \\ \hline
$\{2\}$ & 0.1924 & 0.3464 &  0.2350 & 0.4221     \\ \hline
$\{3\}$ &  8.6637 & 14.816 &   0.4634 & 0.7665    \\ \hline
$\{2,3\}$ & 2.3303 & 4.0498  & 0.5246 & 0.9661      \\ \hline
$\{1,3\}$ & 11.490 & 14.613 &  0.6563 & 1.1425     \\ \hline
$\{1,2\}$ &  0.2333 & 0.3071 &   0.4501 & 0.6253    \\ \hline
\end{tabular}
\vspace{1em}
\caption{Minimum and maximum values of the principal minors of $J$ and $AJ$ attained in the rectangle $\rR$ for $tp = (50, 75, 110)$ and $A$ given as in \eqref{QuantitativeA4barelyP}.}
\label{tbl:minors4barelyPJacobian}
\end{table}

\subsubsection{\textbf{Testing for cases where the Jacobian is a $\bm P-$matrix/ positive quasi-definite/ diagonally dominant everywhere}}
We first checked how often the Jacobian is a $P-$matrix everywhere for integer $tp$ values drawn randomly from the interval $[40,150]$ and having increasing order. 
The probability of finding a tube potential vector $tp$ that leads to everywhere $P-$matrix Jacobian was around 75\%. (Changing the density of iodine did not make much of a difference in this probability.) There was neither a case of the Jacobian being diagonally dominant or positive quasi-definite everywhere, nor an invertible linear transformation that leads to such a Jacobian matrix.

\subsection{Four Materials-Four Measurements Case}
In the following, we used a fixed set of materials (gadolinium, bone, iodine, water) in the said order. We observed similar phenomena as in the previous section. Below we present some representative examples obtained by varying the tube potentials $tp$.

\subsubsection{\textbf{The Jacobian can vanish inside the rectangle $\rR$}}
 The Jacobian can vanish inside $\rR$ with the probability around 12\%. For example, the choice $tp = (90, 120, 135, 150)$ leads to $\textstyle \min_{x \in \rR} \det \J(x) =-0.9 \times 10^{-5}$ and $\textstyle \max_{x \in \rR} \det \J(x) = 1.8 \times 10^{-5}$.

\subsubsection{\textbf{Transforming $\bm I$ into a $\bm P-$function when it is not so}}
We consider $tp = (60, 80, 100, 120)$ as an example.  The resulting ME-CT transform $I$ is not a $P-$function. In view of Proposition \ref{TransformI}, by using $10^8$ random matrices $A$ with det$A$=1, we tested if $AJ(x)$ is a $P-$matrix for all $x \in \rR$. The probability of finding a desired $A$ was around 0.003\%. Among all $A$'s such that $AJ(x)$ is a $P-$matrix for all $x \in \rR$, the bigest injectivity constant obtained was $\mu = 0.0270$ when 
\begin{align}\label{QuantitativeA4notP4x4}
A=
\begin{bmatrix}
    \phantom{-}0.0543  &  -0.3339  &  \phantom{-}1.2065  & -0.7603\\
    \phantom{-}1.0426  &  -1.0288 &   -0.0554  &  \phantom{-}0.6331\\
    -0.3425  &  \phantom{-}1.0820  &  \phantom{-}0.1201  & -0.5054\\
    \phantom{-}0.6584  &  \phantom{-}0.9198  & -0.0548   & \phantom{-}0.9759
\end{bmatrix}.
\end{align}
The extremal values of principal minors of both $J(x)$ and $AJ(x)$ in $\rR$ are listed in table \ref{tbl: minors4notPJacobian4x4}.
\begin{table}[h]
\begin{tabular}{|c|c|c|c|c|}
\hline
\multicolumn{5}{|c|}{tp = {(}60, 80, 100, 120{)}} \\ \hline
\multirow{2}{*}{\begin{tabular}[c]{@{}c@{}}minor\\ assoc. to \end{tabular}} & \multicolumn{2}{c|}{minors of $J$} & \multicolumn{2}{c|}{minors of $AJ$} \\ \cline{2-5} 
& min              & max             & min              & max              \\ \hline
$\O$    & 0.0052 &  0.0107 & 0.0052 &  0.0107    \\ \hline
$\{1\}$ & 0.0058 &  0.0146  & 0.0595 &  0.1166    \\ \hline
$\{2\}$ & -0.2241 & 0.1444 & 0.0281 & 0.0557      \\ \hline
$\{3\}$ &-0.0542 &  -0.0035 &  0.0113 &  0.0287    \\ \hline
$\{4\}$ & 1.0908 &  1.6929 &   0.1005 &  0.1981   \\ \hline
$\{1,2\}$ & 0.0569 & 0.1905  & 0.4490 &  0.5897    \\ \hline
$\{1,3\}$ & 0.0472 &  0.1171  &  0.0783 &  0.1839    \\ \hline
$\{1,4\}$ & 0.4303 &  2.8587 &  1.2697 &  2.5032      \\ \hline
$\{2,3\}$ & -0.1055 &  0.4404 &   0.0767 &  0.1398    \\ \hline
$\{2,4\}$ & -16.547 & 22.886 &  0.4490 &  0.8921    \\ \hline
$\{3,4\}$ & -3.4178 &  -2.1027 &   0.2444 &  0.5642   \\ \hline
$\{2,3,4\}$ & 7.7774 &  15.765 &  1.4104 &  2.3715      \\ \hline
$\{1,3,4\}$ & 1.1876 &  2.4306 &  0.9162 &  1.8671   \\ \hline
$\{1,2,4\}$ & 9.1485 &  12.474 &  4.5312 &  5.5645    \\ \hline
$\{1,2,3\}$ & 0.2197 &  0.2951 &   0.6351 &  0.8954   \\ \hline
\end{tabular}
\vspace{1em}
\caption{Minimum and maximum values of the minors of $J$ and $AJ$ attained in the rectangle $\rR$ for $tp = (60, 80, 100, 120)$ and $A$ given as in \eqref{QuantitativeA4notP4x4}.}
\label{tbl: minors4notPJacobian4x4}
\end{table}
 \subsubsection{\textbf{Transforming a barely $\bm P-$function, into a quantitative $\bm P-$function}}
For instance, the choice $tp = (40, 50, 80, 140)$, leads to a $P-$function $I$ with injectivity constant equal to 0.0003 (see table \ref{tbl: minors4barelyPJacobian4x4}). In this case, the probability of finding an invertible linear transformation that maps $I$ into a $P-$function was around 0.1\%. The biggest injectivity constant was 0.0696, which was obtained by using a transformation with matrix
\begin{align}\label{QuantitativeA4barelyP4x4}
A=
\begin{bmatrix}
     \phantom{-}0.6299  & -0.8295 &    \phantom{-}0.7703  & -0.6612\\
   -0.3839  &  \phantom{-}0.9404  &   \phantom{-}0.2405  & -0.9077\\
   -0.6677  &  \phantom{-}0.4124  &   \phantom{-}0.5842  & -0.0593\\
    \phantom{-}0.0376  &   \phantom{-}0.5939  &  \phantom{-}0.2307  &   \phantom{-}0.5767
\end{bmatrix}.
\end{align}

\begin{table}[h]
\begin{tabular}{|c|c|c|c|c|}
\hline
\multicolumn{5}{|c|}{tp = {(}40, 50, 80, 140{)}} \\ \hline
\multirow{2}{*}{\begin{tabular}[c]{@{}c@{}}minor\\ assoc. to \end{tabular}} & \multicolumn{2}{c|}{minors of $J$} & \multicolumn{2}{c|}{minors of $AJ$} \\ \cline{2-5} 
& min              & max             & min              & max              \\ \hline
$\O$    & 0.2798 &  0.7948 &  0.2798 &  0.7948  \\ \hline
$\{1\}$ & 0.2968 &  0.4538  &  0.7059 &  1.0728      \\ \hline
$\{2\}$ &  1.4425 &  2.1837 &  1.0680 &   1.4188     \\ \hline
$\{3\}$ & 0.2331 &  0.3714 &   0.1778 &  0.2355   \\ \hline
$\{4\}$ & 14.060 &  26.853 &   2.4921 &  4.7378   \\ \hline
$\{1,2\}$ & 0.1342 &  0.5759  &  1.0676 &   1.8373   \\ \hline
$\{1,3\}$ & 0.1706 &  0.3923 &  0.1039 &  0.2877     \\ \hline
$\{1,4\}$ &6.9883 &  15.7200 &  3.3856 &  4.2907    \\ \hline
$\{2,3\}$ &0.4822 &  1.6467 &   0.4741 &   0.8499    \\ \hline
$\{2,4\}$ &61.294 & 128.43 &   7.1367 &   12.656  \\ \hline
$\{3,4\}$ & 0.0101 &  0.4843 &   0.6085 &  0.8669     \\ \hline
$\{2,3,4\}$ & 12.681  &  25.310 &  1.8299  &  3.8828      \\ \hline
$\{1,3,4\}$ & 1.9870 &  4.0498 & 0.3587 &  1.0430    \\ \hline
$\{1,2,4\}$ & 10.955  & 14.123 &  2.5013  &   5.1622    \\ \hline
$\{1,2,3\}$ & 0.2100 &   0.2763 &   0.3938  &   0.5663   \\ \hline
\end{tabular}
\vspace{1em}
\caption{Minimum and maximum values of the minors of $J$ and $AJ$ attained in the rectangle $\rR$ for $tp = (40, 50, 80, 140)$ and $A$ given as in \eqref{QuantitativeA4barelyP4x4}.}
\label{tbl: minors4barelyPJacobian4x4}
\end{table}

 \subsubsection{\textbf{Testing for cases where the Jacobian is a $P-$matrix/ positive quasi-definite/ diagonally dominant everywhere}}
For randomly chosen $tp$ values that are in increasing order, the probability of finding an everywhere $P-$matrix Jacobian was around 4\%. As expected, there was no case of the Jacobian being diagonally dominant or positive quasi-definite everywhere.
\section{Conclusions}
In this paper, we addressed the uniqueness problem in ME-CT by focusing on the nonlinear part of the forward model which maps x-ray transform of material densities to energy-weighted integrals corresponding to different x-ray source energy spectra. 
We proved that the dual-energy CT transform is globally injective on a rectangle provided that the Jacobian determinant is nonvanishing everywhere.
 We presented a sufficient criteria for global injectivity of ME-CT transform using the theory of $P-$functions.
We derived global stability results for ME-CT in the determined as well as overdetermined (with more source energy spectra than the number of materials) cases by introducing the notion of quantitative $P-$function.

Our numerical simulations, which use realistic models of source energy spectra, demonstrated that dual-energy CT problem for (bone, water) material pair is globally injective as long as the tube potentials are different. Nevertheless, for (iodine, water) pair, we encountered a vanishing Jacobian determinant with probability 22\%. Moreover, for (bone, water) material pair, the probability of finding a positive quasi-definite Jacobian matrix was 6\%. However, for (iodine, water) material pair, it was not possible to find tube potentials that lead to everywhere diagonally dominant or positive quasi-definite Jacobian even for lower density of iodine.

 For the case of ME-CT, in all the examples we considered, where the Jacobian determinant remains positive throughout the domain, we observed that it is always possible to find a linear transformation that maps ME-CT transform into a $P-$function, which implies that ME-CT is globally injective as long as it is locally injective at least for the examples considered. However, no tube potentials led to a positive quasi-definite or diagonally dominant Jacobian matrix.

\section{Acknowledgements}
The authors thank Emil Sidky for useful discussions and references. The work of G. Bal was supported in part by NSF Grant DMS-1908736 and ONR Grant N00014-17-1-2096.

\bibliographystyle{siam}
\bibliography{References}
\end{document}